\documentclass[12pt]{amsart}

\usepackage{amsmath,amsfonts,amssymb,amsthm,shuffle,latexsym,mathtools,mathrsfs,multirow}
\usepackage[normalem]{ulem}
\usepackage{enumitem}
\usepackage[usenames, dvipsnames, svgnames]{xcolor}
\usepackage{ytableau}
\usepackage{tikz-cd}
\usepackage{pbox}
\usepackage{soul}
\usepackage{algorithm}
\usepackage{algpseudocode}

\makeatletter
\@namedef{subjclassname@2020}{\textup{2020} Mathematics Subject Classification}
\makeatother

\newtheorem{theorem}{Theorem}[section]
\newtheorem{lemma}[theorem]{Lemma}
\newtheorem{proposition}[theorem]{Proposition}
\newtheorem{corollary}[theorem]{Corollary}
\newtheorem{conjecture}[theorem]{Conjecture}
\newtheorem{question}[theorem]{Question}

\theoremstyle{definition}
\newtheorem{definition}[theorem]{Definition}
\newtheorem{example}[theorem]{Example}

\theoremstyle{remark}
\newtheorem{remark}[theorem]{Remark}
\usepackage{makecell}
\numberwithin{equation}{section}

\newcommand{\Des}{\ensuremath{\mathrm{Des}}}

\newcommand{\SYT}{\ensuremath{\mathsf{SYT}}}

\newcommand{\SShT}{\ensuremath{\mathsf{SShT}}}

\newcommand{\Sym}{\ensuremath{\mathsf{Sym}}}
\newcommand{\QSym}{\ensuremath{\mathsf{QSym}}}

\newcommand{\Peak}{\ensuremath{\mathrm{Peak}}}
\newcommand{\PQSym}{\ensuremath{\mathsf{Peak}}}

\newcommand{\set}{\ensuremath{\mathrm{set}}}

\newcommand{\excise}[1]{}

\newlength\cellsize 

\savebox2{%
\begin{picture}(15,15)
\put(0,0){\line(1,0){15}}
\put(0,0){\line(0,1){15}}
\put(15,0){\line(0,1){15}}
\put(0,15){\line(1,0){15}}
\end{picture}}

\newcommand\cellify[1]{\def\thearg{#1}\def\nothing{}%
\ifx\thearg\nothing\vrule width0pt height\cellsize depth0pt%
  \else\hbox to 0pt{\usebox2\hss}\fi%
  \vbox to 15\unitlength{\vss\hbox to 15\unitlength{\hss$#1$\hss}\vss}}

\newcommand\tableau[1]{\vtop{\let\\=\cr
\setlength\baselineskip{-15000pt}
\setlength\lineskiplimit{15000pt}
\setlength\lineskip{0pt}
\halign{&\cellify{##}\cr#1\crcr}}}

\newlength\Cellsize 

\savebox3{%
\begin{picture}(25,25)
\put(0,0){\line(1,0){25}}
\put(0,0){\line(0,1){25}}
\put(25,0){\line(0,1){25}}
\put(0,25){\line(1,0){25}}
\end{picture}}

\newcommand\Cellify[1]{\def\thearg{#1}\def\nothing{}%
\ifx\thearg\nothing\vrule width0pt height\Cellsize depth0pt%
  \else\hbox to 0pt{\usebox3\hss}\fi%
  \vbox to 25\unitlength{\vss\hbox to 25\unitlength{\hss$#1$\hss}\vss}}

\newcommand\Tableau[1]{\vtop{\let\\=\cr
\setlength\baselineskip{-15000pt}
\setlength\lineskiplimit{15000pt}
\setlength\lineskip{0pt}
\halign{&\Cellify{##}\cr#1\crcr}}}

\newlength\Ccellsize 

\savebox4{%
\begin{picture}(25,25)
\put(0,0){\line(1,0){25}}
\put(0,0){\line(0,1){25}}
\put(25,0){\line(0,1){25}}
\put(0,25){\line(1,0){25}}
\end{picture}}

\newcommand\Ccellify[1]{\def\thearg{#1}\def\nothing{}%
\ifx\thearg\nothing\vrule width0pt height\Ccellsize depth0pt%
  \else\hbox to 0pt{\usebox4\hss}\fi%
  \vbox to 25\unitlength{\vss\hbox to 25\unitlength{\hss$#1$\hss}\vss}}

\newcommand\Ttableau[1]{\vtop{\let\\=\cr
\setlength\baselineskip{-15000pt}
\setlength\lineskiplimit{15000pt}
\setlength\lineskip{0pt}
\halign{&\Ccellify{##}\cr#1\crcr}}}

\usepackage{booktabs}

\makeatletter
\@namedef{subjclassname@2020}{\textup{2020} Mathematics Subject Classification}
\makeatother

\begin{document}


\title[Pattern-Avoiding Peak Functions]{Pattern-Avoiding Peak Functions}

\author[M. Slattery-Holmes]{Matthew Slattery-Holmes}
\address{Department of Mathematics and Statistics, University of Otago, 730 Cumberland St., Dunedin 9016, New Zealand}
\email{slama077@student.otago.ac.nz}


\date{20 Oct 2025}


\maketitle

\begin{abstract}
In 2020, Hamaker, Pawlowski, and Sagan introduced the \textit{pattern quasisymmetric functions}, which are quasisymmetric functions associated with pattern-avoidance classes of permutations, and defined via expansions in fundamental quasisymmetric functions. They determined which subsets of the symmetric group $\mathfrak{S}_3$ index pattern quasisymmetric functions that are symmetric, and showed that these symmetric pattern quasisymmetric functions are also Schur-positive. They then posed the question of when symmetry or Schur $P$-positivity occur for analogous quasisymmetric functions defined in terms of peak functions. In this work we answer this question, that is, we identify precisely which subsets of $\mathfrak{S}_3$ give a \textit{pattern-avoiding peak function} that is symmetric, and give explicit formulas for the positive expansion into the closely-related Schur $Q$-functions.
    
\end{abstract}

\section{Introduction}

In 1911, Schur \cite{schur1911darstellung} introduced what are now known as Schur $P$- and Schur $Q$-functions, in order to study the projective representations of the symmetric and alternating groups. Since then, they have been the focus of much study. The peak functions, which span the peak subalgebra of the quasisymmetric functions, were introduced by Stembridge during his study of enriched $P$-partitions in \cite{stembridge1997enriched}, where he also gives a formula for the positive expansion of Schur $Q$-functions in terms of peak functions. 

We let $\mathfrak{S}_n$ denote the set of permutations of size $n$. If we consider permutations as sequences of the integers 1 through $n$, a \textit{pattern} is a subsequence viewed only in terms of the relative order of its elements. The \textit{avoidance class} of a set of permutations $\Pi$ is the set of permutations that contain no pattern in $\Pi$. For more detailed definitions, see Section \ref{patternavoidancedefn}.

The \textit{pattern quasisymmetric functions}, defined by Hamaker, Pawlowski, and Sagan in \cite{HPS20}, are quasisymmetric functions defined in terms of avoidance classes of permutations. In particular, given some set of permutations $\Pi$, the pattern quasisymmetric function $\mathcal{Q}_n(\Pi)$ is the sum of the fundamental quasisymmetric functions indexed by the descent sets of permutations of size $n$ that avoid $\Pi$. The authors of \cite{HPS20} establish for which avoidance classes $\mathrm{Av}(\Pi)$, for $\Pi \subset \mathfrak{S}_3$, the associated pattern quasisymmetric function is symmetric for all $n$. In these cases, they show that the pattern quasisymmetric functions are also Schur-positive for all $n$, meaning that they expand into the Schur function basis of symmetric functions with only positive coefficients. They then raise the question of when an analogous function $R_n(\Pi)$, defined as the sum of peak functions indexed by peak sets of permutations that avoid $\Pi$, is symmetric or Schur $P$-positive. We call the functions $R_n(\Pi)$ \textit{pattern-avoiding peak functions}. 

We answer this question, determining for which $\Pi\subset\mathfrak{S}_3$ such that $\{123,321\}\not\subset \Pi$, the function $R_n(\Pi)$ is symmetric, and in these cases, we give the expansion into Schur $Q$-functions. As in \cite{HPS20} we impose this additional condition on $\Pi$ since, by the Erd\"os--Szekeres theorem \cite{erdos1935combinatorial}, if $\{123,321\}\subset\Pi$, then $\mathrm{Av}_n(\Pi) = \emptyset$ for all $n\geqslant 5$.

\begin{theorem}\label{MainTheoremText}
If $\Pi \subset \mathfrak{S}_3$ and $\{123,321\}\not \subset \Pi$, the pattern-avoiding peak function $R_n(\Pi)$ is symmetric for all $n$ precisely when $\Pi$ is one of the sets given in the left column of Table 1. In all of these cases $R_n(\Pi)$ is Schur $Q$-positive, and for $n\geqslant 3$, the corresponding expansion of $R_n(\Pi)$ into Schur $Q$-functions is given in the right column of Table \ref{mainthm}. 

\end{theorem}

While full definitions will be given in the following section, in Table \ref{mainthm}, $Q_\lambda$ denotes the Schur $Q$-function indexed by a strict partition $\lambda$, $\ell(\lambda)$ is the length of $\lambda$, and $|\SShT(\lambda)|$ is the number of standard shifted tableaux of shape $\lambda$. We elected to give expansions in terms of Schur $Q$-functions rather than Schur $P$-functions because it simplifies the tableau-based arguments used to prove parts of Theorem \ref{MainTheoremText}. This also answers the question of Schur $P$-positivity, since $Q_\lambda = 2^{\ell(\lambda)}P_\lambda$, where $P_\lambda$ is the Schur $P$-function indexed by $\lambda$. We do not give proofs that $R_n(\Pi)$ is not symmetric for other subsets of $\mathfrak{S}_3$, but in all cases counterexamples can be found for small $n$. We give the expansions of $R_n(\Pi)$ into Schur $Q$-functions for $n\geqslant 3$. It is easy to check that if $\Pi\subset \mathfrak{S}_3$ and $n\in\{1,2\}$, then $R_n(\Pi) = nQ_{(n)}$.

\renewcommand{\arraystretch}{1.5}
\begin{table}[ht]
\centering
\caption{Theorem \ref{MainTheoremText}}
\label{mainthm}

\begin{tabular}{@{} p{0.55\linewidth} l @{}}
\toprule
\textbf{Avoidance class patterns $\boldsymbol{\Pi}$} & $\boldsymbol{R_n(\Pi)}$ \textbf{for} $\boldsymbol{n \geq 3}$ \\
\midrule

$\emptyset$ 
& $\displaystyle \sum_{\lambda\vdash n} 2^{n-\ell(\lambda)}|\SShT(\lambda)|Q_\lambda$ \\
\\
\makecell[l]{%
$\{123\}, \{213\}, \{312\}, \{321\}$}
& $\displaystyle\sum_{k=0}^{\lfloor n/2 \rfloor} \left({n\choose k+1} - {n\choose k-1}\right)Q_{(n-k,k)}$ \\
\\
\makecell[l]{%
$\{213,132\}, \{231,312\}, \{123,132\},$ \\ 
$\{132,312\}, \{213,231\}, \{231,321\}$}
& $\displaystyle\sum_{k=0}^{\lfloor n/2 \rfloor} (n-2k)Q_{(n-k,k)}$ \\
\\
$\{132,231\}$ 
& $2^{n-1}Q_{(n)}$ \\
\\
\makecell[l]{%
$\{123,132,312\}, \{123,213,231\}, \{132,312,321\},$ \\ 
$\{132,213,321\}, \{132,213,312\}, \{123,231,312\},$ \\ 
$\{213,231,321\}, \{213,231,312\}$}
& $2Q_{(n)}+Q_{(n-1,1)}$ \\
\\
\makecell[l]{%
$\{123,132,231\}, \{132,213,231\},$ \\ 
$\{132,231,312\}, \{132,231,321\}$}
& $nQ_{(n)}$ \\
\\
\makecell[l]{%
$\{123,132,213,231\}, \{123,132,231,312\},$ \\ 
$\{132,213,231,312\}, \{132,213,231,321\},$ \\ 
$\{132,231,312,321\}$}
& $2Q_{(n)}$ \\
\\
\makecell[l]{%
$\{123,132,213,231,312\},$ \\ 
$\{132,213,231,312,321\}$}
& $Q_{(n)}$ \\

\bottomrule
\end{tabular}
\end{table}

\section{Background}
Given a positive integer $n$, we denote the set $\{1,2,\ldots,n\}$ by $[n]$. A \textit{partition} $\gamma$ of $n$, written $\gamma \vdash n$, is a weakly decreasing sequence of positive integers that sum to $n$. A partition whose elements are strictly decreasing is called a \textit{strict partition}, and we will denote these by $\lambda$. If $\gamma= (\gamma_1,\gamma_2,\ldots,\gamma_k)$, then its \textit{length} is $\ell(\gamma)=k$. 

Given a partition $\gamma$, the \textit{Young diagram of shape $\gamma$} is the left-justified array of cells with $\gamma_1$ cells in the first (lowest) row, $\gamma_2$ in the second, and so on. For a strict partition $\lambda$, the  \textit{shifted diagram of shape $\lambda$} is obtained by taking the Young diagram of shape $\lambda$, and indenting the $i^{th}$ row one space to the right relative to the $(i-1)^{th}$ row. 

A \textit{standard Young tableau}, or $\SYT$, is a filling of the Young diagram of shape $\gamma \vdash n$ with numbers $1,2,\ldots, n$ such that rows increase from left to right, and columns increase from bottom to top. The set of $\SYT$ of shape $\gamma$ is referred to as $\SYT(\gamma)$.

A \textit{standard shifted tableau}, or $\SShT$, is a filling of the shifted diagram of a strict partition $\lambda \vdash n$ with the numbers $1,2,\ldots,n$ such that rows increase from left to right, and columns increase from bottom to top. The set of $\SShT$ of shape $\lambda$ is called $\SShT(\lambda)$.

The \textit{descent set} of a tableau filling $T\in \SYT(\lambda)$ or $T\in \SShT(\lambda)$ is defined as
\[
\Des(T) = \{\,i : i+1 \text{ is strictly above } i \,\}.
\]

Given a set $S$ of integers, the \textit{peak set} of $S$ is the set \[\Peak(S) = \{s\in S: s\neq 1, s-1\notin S\}.\] We will frequently refer to the peak set of the descent set of a standard shifted tableau $T$, so we adopt the abbreviated notation $\Peak(T):=\Peak(\Des(T))$. We refer to the elements of $\Peak(T)$ as the \textit{peaks} of $T$, and $\Peak(T)$ itself as the \textit{peak set} of $T$.

\begin{example}
Let $\lambda = (3,2,1)$. The Young diagram and shifted diagram of shape $\lambda$ are, respectively, \[\begin{array}{c@{\hskip \cellsize}c@{\hskip \cellsize}c@{\hskip \cellsize}c@{\hskip \cellsize}c@{\hskip \cellsize}c}\tableau{~\\~&~\\~&~&~}\end{array}\!\text{ and }~~\begin{array}{c@{\hskip \cellsize}c@{\hskip \cellsize}c@{\hskip \cellsize}c@{\hskip \cellsize}c@{\hskip \cellsize}c}\tableau{&&~\\&~&~\\~&~&~}\end{array}\!\!\!\!\!.\]
The standard shifted tableaux of shape $(3,2,1)$ are the fillings \[\SShT(3,2,1) = \left\{\begin{array}{c@{\hskip \cellsize}c@{\hskip \cellsize}c@{\hskip \cellsize}c@{\hskip \cellsize}c@{\hskip \cellsize}c}\tableau{&&6\\&4&5\\1&2&3} \end{array}\!\!\!\!\!\!~,~\!\!\!\begin{array}{c@{\hskip \cellsize}c@{\hskip \cellsize}c@{\hskip \cellsize}c@{\hskip \cellsize}c@{\hskip \cellsize}c}\tableau{&&6\\&3&5\\1&2&4} \end{array}\!\!\!\!\!\!\right\}\] which have descent sets $\{3,5\}$ and $\{2,4,5\}$, and peak sets $\{3,5\}$ and $\{2,4\}$, respectively. 

\end{example}

\subsection{Quasisymmetric and symmetric functions}

A \textit{composition} $\alpha$ of $n$ is a finite sequence of positive integers that sums to $n$, and we denote this by $\alpha\vDash n$. The algebra of formal power series in countably many variables is denoted $\mathbb{C}[[x_1,x_2,\ldots]]$. The algebra of \textit{quasisymmetric functions}, denoted $\QSym$, is defined as those elements $f \in \mathbb{C}[[x_1,x_2,\ldots]]$ such that given any composition $(\alpha_1,\alpha_2,\ldots,\alpha_k)$, the coefficient of the term $x_1^{\alpha_1}\cdots x_k^{\alpha_k}$ in $f$ is equal to the coefficient of $x_{i_1}^{\alpha_1}\cdots x_{i_k}^{\alpha_k}$ in $f$, for all index sequences $i_1<i_2<\cdots<i_k$ of positive integers. The \textit{symmetric functions} are a subalgebra of $\QSym$, denoted $\Sym$, defined as all functions $f\in \mathbb{C}[[x_1,x_2,\ldots]]$ such that $f$ is unchanged under permuting the variables. 

The \textit{monomial quasisymmetric functions} \cite{gessel1984multipartite} form a basis for $\QSym$. Given any composition $\alpha = (\alpha_1, \ldots , \alpha_k)\vDash n$, the monomial quasisymmetric function $M_\alpha$ is 
\[M_\alpha = \sum_{1\le i_1<i_2<\cdots <i_k}x_{i_1}^{\alpha_1}x_{i_2}^{\alpha_2}\cdots x_{i_k}^{\alpha_k}.\]

Given compositions $\alpha$ and $\beta$, we say $\beta$ \textit{refines} $\alpha$ if one can produce $\alpha$ by adding together adjacent parts of $\beta$. Given a composition $\beta = (\beta_1,\beta_2,\ldots,\beta_k)\vDash n$, let $\set(\beta) = \{\beta_1,\beta_1+\beta_2,\ldots,\beta_1+\cdots+\beta_{k-1}\}\subseteq [n-1]$. 

Another important basis for $\QSym$ that we will consider in this work is the \textit{fundamental quasisymmetric functions} \cite{gessel1984multipartite}. For a composition $\alpha\vDash n$, the fundamental quasisymmetric function corresponding to $\alpha$ is \[F_\alpha = \sum_\beta M_\beta,\] where the sum is taken over all $\beta$ that refine $\alpha.$

A set of integers $S$ is said to be a \textit{peak set} if $1\notin S$ and for every $i\in S$, $i-1\notin S$. The \textit{peak algebra} $\PQSym$ is a subalgebra of $\QSym$. Bases of $\PQSym$ are indexed by peak sets, and one notable basis is the \textit{peak functions} \cite{stembridge1997enriched} which for peak sets $S\subset[n-1]$ are given by \[K_{S,n} = 2^{|S|+1} \sum_{\substack{ \beta \vDash n \\ S \subset \set(\beta) \triangle \big(\set(\beta)+1\big)}} F_\beta,\] where $\triangle$ represents the symmetric difference of sets, and $\set(\beta)+1$ is shorthand for $\{x+1:x\in \set(\beta)\}$. When the degree of a peak function is clear from the context, we drop the $n$ and simply write $K_{S}$.

The \textit{Schur $Q$-functions} span the intersection of $\PQSym$ and $\Sym$. Given a strict partition $\lambda \vdash n$, the expansion of the Schur $Q$-function $Q_\lambda$ in terms of peak functions is \cite{assaf2014shifted,stembridge1997enriched}
\begin{equation}\label{SchurQdefn}
    Q_\lambda = \sum_{T\in \SShT(\lambda)}K_{\Peak(T)}.
\end{equation}

\subsection{Pattern avoidance in permutations}\label{patternavoidancedefn}
The symmetric group of permutations of $[n]$ is denoted $\mathfrak{S}_n$, and we say $\pi\in \mathfrak{S}_n$ is a permutation of \textit{size} $n$. A permutation $\pi\in \mathfrak{S}_n$ is written in one-line notation as $\pi = \pi_1\pi_2\ldots\pi_n$, where $\pi_i= \pi(i)$ is the $i^{th}$ \textit{element} of $\pi$. We denote the \textit{decreasing permutation} of size $j$ by $\delta_j$, and the \textit{increasing permutation} by $\iota_j$, i.e. $\delta_j = j(j-1)\ldots 21$ and $\iota_j = 12\ldots(j-1)j$. 

Given a permutation $\pi = \pi_1\pi_2\ldots\pi_n$, the \textit{reverse} permutation $\pi^r$ is $\pi_n\pi_{n-1}\ldots\pi_1$. Given a collection of permutations $\Pi$, we let  $\Pi^r = \{\pi^r:\pi\in \Pi\}$. 

The \textit{descent set of a permutation} is the set of indices \[\Des(\pi) = \{i:\pi_i > \pi_{i+1}\}.\] As with tableaux, we denote the peak set of the descent set of a permutation $\pi$ by $\Peak(\pi)$ rather than $\Peak(\Des(\pi))$. We call $\Peak(\pi)$ the \textit{peak set} of $\pi$, and refer to the elements of $\Peak(\pi)$ as the \textit{peaks} of $\pi$. In order to reduce excessive notation we have elected to use $\Des$ and $\Peak$ to refer to descent sets and peak sets respectively of both permutations and tableaux. From context it should be clear which one is meant.

Given a permutation $\pi = \pi_1\pi_2\cdots\pi_n$, let $\rho = \pi_{i_1}\cdots \pi_{i_k}$ denote some subsequence of size $k\leq n$. We say $\rho$ \textit{standardizes} to a permutation $st(\rho)$ of size $k$ if $st(\rho)$ is the permutation we obtain when we replace the smallest element of $\rho$ with a 1, the second smallest with a 2, and so on. In this manner, we say that $\pi$ \textit{contains} the permutation $st(\rho)$ as a \textit{pattern}. Given a permutation $\pi$ of size $n$ and a permutation $\tau$ of size $k$, we say $\pi$ \textit{avoids} $\tau$ if $\pi$ does not contain $\tau$.

\begin{example}
The permutation $\sigma = 316245$ contains the pattern 132 via the subsequences 364, 365, and 162. It avoids the pattern $\delta_3$. It has descent set $\Des(\sigma) = \{1,3\}$ and peak set $\Peak(\sigma) = \{3\}$. The reverse is $\sigma^r = 542613$, and $\Peak(\sigma^r) = \{4\}$.  
\end{example}
The subject of permutation pattern containment and avoidance has been the focus of a great deal of research in recent years, see for example \cite{Bon08,Kit11} and the references therein.

Given a collection of permutations $\Pi$, the \textit{avoidance class}, $\mathrm{Av}(\Pi)$, is the set of permutations of any size that do not contain any pattern in $\Pi$, and the set of all permutations in $\mathrm{Av}(\Pi)$ that have size $n$ is denoted $\mathrm{Av}_n(\Pi)$. Two avoidance classes (equivalently, two sets of permutations) are said to be \textit{Wilf-equivalent} if the avoidance classes contain the same number of permutations of size $n$ for all $n$. That is, the sets of permutations $\Pi$ and $\Pi'$ are Wilf-equivalent if $|\mathrm{Av}_n(\Pi) | = |\mathrm{Av}_n(\Pi')|$ for all $n$. 

\subsection{Pattern-avoiding peak functions}

\begin{definition}\cite{HPS20}
Given a set of permutations $\Pi$, the \textit{pattern-avoiding peak function} of degree $n$ corresponding to $\Pi$ is \[R_n(\Pi) = \sum_{\pi\in \mathrm{Av}_n(\Pi)}K_{\Peak(\pi)}.\] 
\end{definition}

\begin{example}
For $\Pi = \{1234,1243,2413,3142,3412,4123\}$ the interested reader may use a computer to verify that \[R_6(\Pi) = 10Q_{(6)}+12Q_{(5,1)}+8Q_{(4,2)}-Q_{(3,2,1)}.\]
This is an example of a pattern-avoiding peak function that is symmetric but not Schur $Q$-positive.

\end{example}

\subsection{Insertion algorithms}

In this work we will make use of two insertion algorithms that map a word $w = w_1\ldots w_k$ consisting of positive integers to pairs of tableaux. The first is \textit{Robinson--Schensted--Knuth insertion}, commonly abbreviated to $\mathrm{RSK}$ insertion \cite{knuth1970permutations,robinson1938representations,schensted1961longest}, which gives a bijection between elements of $\mathfrak{S}_n$ and pairs of standard Young tableaux with $n$ cells, of the same shape: \[\mathrm{RSK}:\pi\mapsto (P(\pi),Q(\pi)),\] where $P(\pi)$ is called the insertion tableau, and $Q(\pi),$ the recording tableau.

We will now describe $\mathrm{RSK}$ insertion on permutations. Let $\pi = \pi_1\pi_2\cdots \pi_n$ be a permutation of size $n$. We construct the pair of tableaux $(P(\pi),Q(\pi))$ using the $\mathrm{RSK}$ insertion algorithm, as follows. 

We begin with empty tableaux $P_0$ and $Q_0$. 

For $i\in [n]$, we create $P_i$ from $P_{i-1}$ by inserting $x = \pi_i$ using \textit{row insertion}. The process begins with row insertion of $x$ into row 1. To row insert $x$ into some row, locate the smallest (leftmost) entry $y$ in that row such that $y>x$. If no such $y$ is found, append $x$ to the rightmost end of the row, creating a new cell. If such an entry $y$ is found, replace $y$ with $x$, (which we call \textit{bumping} $y$). Then proceed to insert $y$ into the row immediately above that which it was bumped from. 

Once an entry is placed into a new cell at the rightmost end of a row, this new tableau is $P_i$. Create the new recording tableau $Q_i$ from $Q_{i-1}$ by adding a cell containing $i$ to $Q_{i-1}$ in the same position as the the new cell that was created in $P_i$. 

After inserting $\pi_i$ for all $i\in [n]$ we set 
    \[
    P(\pi)=P_n \quad \text{and} \quad Q(\pi)=Q_n.
    \]

The other insertion algorithm that we will use is \textit{Sagan--Worley insertion} \cite{sagan1987shifted,worley1984theory}. This gives a bijection between $\mathfrak{S}_n$ and pairs of shifted tableaux $(R(\pi),S(\pi))$ with $n$ cells, of the same shape, where the insertion tableau $R(\pi)$ is a standard shifted tableau, and the recording tableau $S(\pi)$ is a \textit{marked standard shifted tableau} of the same shape as $R(\pi)$, i.e. a standard shifted tableau whose entries off the main diagonal may be marked or unmarked. In the Sagan--Worley insertion algorithm, there are two types of insertion, a row insertion analogous to that of $\mathrm{RSK}$  insertion, and a column insertion that occurs after an entry has been bumped from the main diagonal. We describe the Sagan--Worley insertion algorithm to construct $(R(\pi),S(\pi))$ now.

Let $\pi = \pi_1\pi_2\cdots\pi_n$ be a permutation of size $n$. Begin with empty tableaux $R_0$ and $S_0$. For $i\in [n]$, create $R_i$ from $R_{i-1}$ by inserting $x = \pi_i$. The process begins with \textit{row insertion}, inserting $x$ into row 1 by locating the smallest entry $y>x$. If no such $y$ is found, append $x$ to the end of the row, creating a new cell. If such an entry $y$ is found, replace $y$ with $x$, which is called \textit{bumping} $y$. If $y$ was bumped from any cell other than the leftmost cell in its row (i.e. not on the main diagonal of $P_{i-1}$), then row insert $y$ into the next higher row. We refer to bumping that is followed by row insertion as \textit{row bumping}.

If $y$ was bumped from the main diagonal, we switch to \textit{column insertion}. Scan the column immediately to the right of the one which $y$ was bumped from, searching for the smallest entry $z$ in this column such that $z>y$. If no such $z$ exists, place $y$ into a new cell at the top of the column. If such a $z$ exists, bump $z$ and replace it with $y$, then column insert $z$ into the next column to the right. We refer to bumping that is followed by column insertion as \textit{column bumping}.

Once an entry is placed into a newly created cell, the shifted tableau created is $R_i$. Create the new recording tableau $S_i$ from $S_{i-1}$ by adding a cell containing $i$ to $S_{i-1}$ in the position of the new cell created in $R_i$. If column insertion occurred at any point between the insertion of $\pi_i$ and creating a new cell, then we mark the entry $i$ in the recording tableau. 

After inserting $\pi_i$ for all $i\in [n]$ we set 
    \[
    R(\pi)=R_n \quad \text{and} \quad S(\pi)=S_n,
    \]
and we say \[\mathrm{SW}:\pi\mapsto(R(\pi),S(\pi)).\]

\begin{example}
    Consider the permutation $\pi = 4612537\in \mathrm{Av}_7(321)$. Performing $\mathrm{RSK}$ insertion on $\pi$ we get \[P(\pi) =  \begin{array}{c@{\hskip \cellsize}c@{\hskip \cellsize}c@{\hskip \cellsize}c@{\hskip \cellsize}c@{\hskip \cellsize}c}\tableau{6\\4&5\\1&2&3&7}\end{array} \!\!\!\!\!,~~ Q(\pi) = \begin{array}{c@{\hskip \cellsize}c@{\hskip \cellsize}c@{\hskip \cellsize}c@{\hskip \cellsize}c@{\hskip \cellsize}c} \tableau{6\\3&4\\1&2&5&7}\end{array}\!\!\!\!\!.\]
    Performing Sagan--Worley insertion on $\pi$ we get \[R(\pi) = \begin{array}{c@{\hskip \cellsize}c@{\hskip \cellsize}c@{\hskip \cellsize}c@{\hskip \cellsize}c@{\hskip \cellsize}c}\tableau{&&6\\&4&5\\1&2&3&7}\end{array} \!\!\!\!\!,~~ S(\pi) =  \begin{array}{c@{\hskip \cellsize}c@{\hskip \cellsize}c@{\hskip \cellsize}c@{\hskip \cellsize}c@{\hskip \cellsize}c}\tableau{&&6\\&4&5\\1&2&3'&7}\end{array}\!\!\!\!\!.\]
    
\end{example}

\section{Proof of Theorem \ref{MainTheoremText}}
In this section we will prove Theorem \ref{MainTheoremText}. We proceed through Table 1 row by row, beginning with the top row. 

\subsection{The empty set}

\begin{proposition}\label{insertionpeaks}
When performing Sagan--Worley insertion on a permutation $\pi$, the peaks of the insertion tableau created are precisely those elements $p$ of $\pi$ such that $p$ is right of both $p-1$ and $p+1$ in $\pi$. 
\end{proposition}
\begin{proof}

Let $P = \{\pi_i: \pi_i$ is right of both $\pi_i-1$ and $\pi_i+1$ in $\pi$\}. We will verify that for all $p\in P$, $p$ is weakly below $p-1$ and strictly below $p+1$ (which we shall refer to as being in \textit{peak position}) in the insertion tableau obtained by inserting $p$, that $p$ remains in peak position following all subsequent insertions, and that this condition is not met for any $q\notin P$. When performing Sagan--Worley insertion on $\pi$, we will insert both $p-1$ and $p+1$ before $p$. This means that $p+1$ cannot be in the first position of the bottom row when $p$ is inserted, because if an entry larger than $p-1$ was here when $p-1$ was inserted, $p-1$ would bump that entry. If $p+1$ is still in the bottom row when $p$ is inserted, $p$ will bump it to a higher row. Thus, immediately after inserting $p$, $p$ is in peak position.

Next, we will consider what may happen if $p$ is bumped while weakly below $p-1$ and strictly below $p+1$. It will be useful to note that when column insertion occurs, an entry that is bumped inserts into the next column to the right, which also implies that that column insertion can never increase the row index of a bumped entry. If $p$ is bumped by column insertion, then $p$ will place into a weakly lower row, and remain in peak position. If $p$ is bumped by row insertion, then the entry that bumped $p$ did not bump $p-1$, meaning that $p-1$ must already be in a higher row. This means that the leftmost entry of the row $p$ bumps into must be less than or equal to $p-1$. In particular, if $p+1$ is in this row, it is not in the first position, and $p$ will bump it to a higher row. Therefore, if $p$ is row bumped, it remains in peak position. 

If $p-1$ is bumped by row insertion, it will place into a higher row, which doesn't create any issue. Suppose $p-1$ is bumped by column insertion. Prior to being bumped, $p-1$ is either in the same row as $p$, or it is strictly above $p$ and strictly to the left. If $p-1$ is in the same row as $p$, it will column bump $p$, which will move right and weakly downwards, remaining in peak position. If $p-1$ is strictly above $p$ when it is bumped, it will bump into the next column, in the position of the least entry larger than $p-1$. Any entry weakly below and weakly left of $p$ is smaller than $p$, so $p-1$ cannot insert here. It also cannot insert below and strictly right of $p$, since this would require $p-1$ to be above and weakly right of $p$ before being bumped. Hence, if $p-1$ is bumped, $p$ remains in peak position. 

Suppose $p+1$ is bumped. Again, if it is bumped by row insertion, it will obviously remain above $p$. If $p+1$ is bumped by column insertion and it was strictly left of $p$ before being bumped, it will remain above $p$ since no entry weakly below and weakly left of $p$ is larger than $p+1$. If $p+1$ is weakly right of $p$ and strictly above $p$, then it can only be in the cell directly above $p$. In this case, $p+1$ cannot be column bumped since the entry that bumped $p+1$ would be smaller than $p+1$ and therefore smaller than $p$, and $p$ would be bumped before $p+1$. Note that $p$ cannot column bump $p+1$ as $p$ is strictly below $p+1$ at all times, so again, $p$ remains in peak position. We have now shown that entries $p$ such that $p-1$ and $p+1$ are left of $p$ in $\pi$, are sent to peaks in the Sagan--Worley insertion tableau $R(\pi)$. 

Next we must show that no $q\notin P$ may be a peak of $R(\pi)$. There are two cases to consider. Either $q$ is left of $q+1$ in $\pi$, or $q$ is left of $q-1$ and right of $q+1$. 

In the first case, we claim that $q+1$ will always remain weakly below $q$ in the sequence of insertion tableaux, during the insertion process. Certainly $q$ inserts before $q+1$ and so $q+1$ will initially insert weakly below $q$. Also, $q+1$ can never be bumped to a position strictly above $q$ because if $q+1$ is ever bumped by row insertion, then $q$ must be strictly above $q+1$. If $q+1$ is bumped by column insertion, it will move weakly lower and $q$ will remain weakly above $q+1$. If $q$ is weakly above $q+1$, it must be strictly to the left of $q+1$. This means that if $q$ is bumped by column insertion, it will remain weakly above $q+1$, since all entries strictly below and left of $q+1$ are smaller than $q$, and $q$ cannot bump to the right of $q+1$ by column insertion. Hence, $q$ can never find itself strictly below $q+1$. 

For the second case, $q+1$ is inserted before $q$, which is inserted before $q-1$. We consider two subcases:
First, suppose $q+1$ is in the first cell of the bottom row when $q$ is inserted. In this subcase, $q$ will column bump $q+1$, and so after inserting $q$, $q$ is not strictly below $q+1$. If $q+1$ is ever row bumped from the first row, then $q$ must have already been bumped to a higher row, so $q+1$ cannot be row bumped to a higher position than $q$. If $q$ is ever column bumped, it must be strictly left of $q+1$, and can never insert strictly below $q+1$, since all entries weakly left and strictly below $q+1$ are smaller than $q$. Hence, $q$ remains weakly above $q+1$. 

For the second subcase, we assume $q+1$ is not in the first position of the first row when $q$ is inserted. This means that either $q$ will bump $q+1$, or $q+1$ is already higher than the first row when $q$ is inserted. It also means that $q$ does not insert into the first position of the first row, since the entry here must be smaller than $q+1$. When $q-1$ is inserted, either it will bump $q$, or $q$ is already above the first row. Therefore, immediately after inserting $q-1$, $q-1$ is strictly below $q$. 

Suppose $q-1$ is row bumped to the row containing $q$. If $q$ is not in the first position, $q-1$ will bump $q$ to a higher row, and $q$ will not be in peak position. If $q$ is in the first position, then before $q$ inserted to the first position of this row, everything weakly above this position was greater than $q$, which implies that $q$ must have bumped $q+1$ from this position. This means that $q-1$ will column bump $q$ one position to the right, and $q$ will column bump $q+1$, which will move weakly downwards. From here, $q$ is protected from entering peak position, since either $q+1$ is immediately to the right of $q$, or $q+1$ is below $q$. In the first scenario, the entry below $q+1$ is smaller than $q$, so if $q$ is column bumped, it will bump $q+1$ rather than dropping to a lower row. In the second scenario, $q+1$ must be below and right of $q$, so if $q$ is column bumped to a lower row it will not insert below $q+1$. Hence $q$ is never bumped to peak position. Therefore, given a permutation $\pi$, the elements of $\pi$ which become peaks in the (Sagan--Worley) insertion tableau $R(\pi)$ are precisely those elements $p$ such that $p-1$ and $p+1$ are left of $p$ in $\pi$. 
\end{proof}

\begin{corollary}\label{emptythm}
Let $\mathrm{SW}:\pi^{-1} \mapsto (R(\pi^{-1}),S(\pi^{-1}))$. Then $\Peak(\pi) = \Peak(R(\pi^{-1}))$. 
\end{corollary}

\begin{proof}
Let $p\in \Peak(\pi)$. Then we have $\pi_{p-1} < \pi_p > \pi_{p+1}$, which gives that $p$ is right of both $p-1$ and $p+1$ in $\pi^{-1}$. Then, by Proposition \ref{insertionpeaks}, we have $p\in \Peak(R(\pi^{-1}))$. 
Next, suppose $q\notin \Peak(\pi)$. Either $\pi_q < \pi_{q+1}$, i.e. $q$ is not a descent, or $\{q-1,q\}\subseteq \Des(\pi)$, so $\pi_{q-1} > \pi_q > \pi_{q+1}$. In the first case $q+1$ is right of $q$ in $\pi^{-1}$, and in the second case, $q-1$ is to the right of $q$, which is to the right of $q+1$ in $\pi^{-1}$. In both situations, again by Proposition \ref{insertionpeaks}, we see $q\notin \Peak(R(\pi^{-1}))$.
\end{proof}

\begin{theorem}
\[R_n(\emptyset) = \sum_{\lambda} 2^{n-\ell(\lambda)}|\SShT(\lambda)|Q_\lambda.\]
\end{theorem}
\begin{proof}

For a marked standard shifted tableau $S$, we let $\mathrm{unmark}(S)$ denote the tableau obtained by removing all of the markings on entries of $S$. For some standard shifted tableau $T$ with $n$ cells, let $A(T)$ denote the set of permutations whose inverse maps to any recording tableau that `unmarks' to $T$ under Sagan--Worley insertion, i.e. \[A(T) = \{\pi: \mathrm{unmark}(S(\pi^{-1})) = T\}.\]
Let $\lambda\vdash n$ be the shape of $T$. Only the off-diagonal entries of a marked standard shifted tableau may be marked (recall marked entries are those bumped \textit{from} the main diagonal, and in such an instance, the new cell created is always strictly right of, and weakly below, the cell where column insertion began). As such, there are $2^{n-\ell(\lambda)}$ marked $\SShT$ of shape $\lambda$ that unmark to $T$. Given a marked $\SShT$ of shape $\lambda$, there are $|\SShT(\lambda)|$ permutations for which that marked $\SShT$ is the recording tableau. Therefore, $A(T)$ contains $2^{n-\ell(\lambda)}|\SShT(\lambda)|$ permutations. The peak sets of these permutations are in bijection with peak sets of (Sagan--Worley) insertion tableaux that pair with a recording tableau that unmarks to $T$. If we sum the peak functions associated with these permutations, we obtain the following:
\begin{equation}\label{A(T)}\sum_{\pi \in A(T)} K_{\Peak(\pi)} = \sum_{P \in \SShT(\lambda)}2^{n-\ell(\lambda)} K_{\Peak(P)} = 2^{n-\ell(\lambda)} Q_\lambda.\end{equation}

Let $A(\lambda)$ denote the set of permutations whose inverse maps under Sagan--Worley insertion to any pair of tableaux of shape $\lambda$, i.e.  \[A(\lambda) = \bigcup_{T\in \SShT(\lambda)}A(T).\] 
If we sum the peak functions associated with each permutation in $A(\lambda)$, we get the following:

\begin{align*}\sum_{\pi\in A(\lambda)}K_{\Peak(\pi)} &= \sum_{T\in \SShT(\lambda)}\sum_{\pi\in A(T)} K_{\Peak(\pi)}\\[6 pt] &= \sum_{T\in \SShT(\lambda)} 2^{n-\ell(\lambda)}Q_\lambda\\[6 pt] &= 2^{n-\ell(\lambda)}|\SShT(\lambda)| Q_\lambda,  \end{align*}

where the second equality is from \eqref{A(T)}. Summing over all $\lambda\vdash n$, we obtain

\[R_n(\emptyset) =  \sum_{\lambda \vdash n}\sum_{\pi\in A(\lambda)}K_{\Peak(\pi)} = \sum_{\lambda \vdash n} 2^{n-\ell(\lambda)}|\SShT(\lambda)|Q_\lambda.\]

\vspace{-0.5cm}
\end{proof}

\subsection{The single pattern cases.}

In Table \ref{mainthm}, most of the formulae given are true for several different Wilf-equivalent sets of patterns. We introduce a stronger notion of equivalence on sets of permutations. 

\begin{definition}
Given two sets of permutations, $\Pi$ and $\Pi'$, we say that these sets are \textit{peak-equivalent} if for all $n$, \[\big\{\!\!\big\{\Peak(\pi):\pi\in \mathrm{Av}_n(\Pi)\big\}\!\!\big\} = \big\{\!\!\big\{\Peak(\pi):\pi\in \mathrm{Av}_n(\Pi')\big\}\!\!\big\},\] where the double brackets denotes a multiset. Clearly $\Pi$ and $\Pi'$ are peak-equivalent if and only if $R_n(\Pi) = R_n(\Pi')$ for all $n$. 
\end{definition}

We will now show that if a pattern-avoiding peak function $R_n(\Pi)$ is symmetric, then $\Pi$ and $\Pi^r$ are peak-equivalent, which reduces the number of cases we must check for many rows in Table \ref{mainthm}.

\begin{lemma}\label{reverselemma}
Given some collection of patterns $\Pi$, if the pattern-avoiding peak function $R_n(\Pi)$ is symmetric, then $\Pi$ and $\Pi^r$ are peak-equivalent. In particular, if $R_n(\Pi)$ is symmetric, then \[ R_n(\Pi) = R_n(\Pi^r) .\]
\end{lemma}

\begin{proof}
We will demonstrate that, upon expanding $R_n(\Pi)$ into monomial quasisymmetric functions, reversing the permutations in the set of patterns $\Pi$ is equivalent to reversing the compositions which index the monomial quasisymmetric functions, i.e. \[R_n(\Pi) = \sum_\alpha c_\alpha M_\alpha \implies R_n(\Pi^r) = \sum_\alpha c_\alpha M_{\alpha^r}.\] Fix $n$, choose some permutation $\pi\in \mathfrak{S}_n$, and observe that \[\Peak(\pi^r) = \{n+1-p_i~|~p_i\in \Peak(\pi)\}.\]
If $F_\beta$ appears in the expansion of $K_{\Peak(\pi)}$ into fundamental quasisymmetric functions for some $\beta = (\beta_1,\ldots,\beta_j)$, then for all $p_i\in \Peak(\pi)$, we have $p_i\in \set(\beta)\triangle \big(\set(\beta)+1\big).$ We wish to show that \[(n+1-p_i) \in \set(\beta^r)\triangle \big(\set(\beta^r)+1\big).\]

Note that we can express the set of the reverse of the composition $\beta$ as \[\set(\beta^r) = \left\{n-\sum_{i=1}^{j-1}\beta_i,n-\sum_{i=1}^{j-2}\beta_i,\ldots, n-\beta_1\right\}. \]

If some element $p$ is in $ \set(\beta)\triangle \big(\set(\beta)+1\big)$ then there are two cases to consider. 
If $p\in \set(\beta)$ and $p\notin \big(\set(\beta)+1\big)$, then $p = \sum_{i=1}^g \beta_i$ for some $g$, and $p \neq \big(\sum_{i=1}^{h-1} \beta_i\big) + 1$ for any $h$. 

This gives \[n-p+1 = n-\sum_{i=1}^g \beta_i +1 ~~\text{  and  } ~~n-p+1 \neq  n-\sum_{i=1}^h \beta_i\] so we see $n+1-p \in  \big(\set(\beta^r)+1\big)$ and $n+1-p\notin \set(\beta^r)$. In other words, $p\in \set(\beta^r)\triangle(\set(\beta^r)+1)$. The argument when $p\notin \set(\beta)$ and $p\in \set(\beta)+1$ is similar so we will omit it here. Additionally, if there is some element $q\in [n-1]$ such that $q\notin \set(\beta)\triangle(\set(\beta^r)+1)$, then either $q$ cannot be written as $\sum_{i=1}^g\beta_i$ or $\sum_{i=1}^{h-1}\beta_i+1$ for any $g,h$, or there is some $g$ such that $q = \sum_{i=1}^g\beta_i = \sum_{i=1}^{g-1}\beta_i +1$. In either case it is straightforward to check that $n+1-q \notin \set(\beta^r)\triangle(\set(\beta^r)+1)$, and so $\Peak(\pi) \subseteq \set(\beta)\triangle (\set(\beta)+1) $ implies  $\Peak(\pi^r) \subseteq \set(\beta^r)\triangle (\set(\beta^r)+1) $. Hence, since 
\begin{equation}\label{peakfb1}K_{\Peak(\pi)} = 2^{|\Peak(\pi)|}\sum_{\Peak(\pi)\subseteq \set(\beta)\triangle (\set(\beta)+1)} F_\beta, \end{equation}  upon reversing the permutation $\pi$, we have \begin{equation}\label{peakfb2}
    K_{\Peak(\pi^r)} = 2^{|\Peak(\pi)|}\sum_{\Peak(\pi)\subseteq \set(\beta^r)\triangle (\set(\beta^r)+1)} F_{\beta^r}.\end{equation}
Note that the coefficients in \eqref{peakfb1} and \eqref{peakfb2} only depend on the size of the indexing peak set, and $|\Peak(\pi)| = |\Peak(\pi^r)|$ so the coefficients are the same. 

To further expand into monomial quasisymmetric functions, recall that $F_\beta = \sum_\alpha M_\alpha$ where the sum is taken over all $\alpha$ that refine $\beta$. Clearly $\alpha$ refines $\beta$ if and only if $\alpha^r$ refines $\beta^r$, so we have \[K_{\Peak(\pi)} = \sum_\alpha c_\alpha M_\alpha \implies K_{\Peak(\pi^r)} = \sum_\alpha c_\alpha M_{\alpha^r}.\]  
In other words, we have shown that reversing the permutation indexing a peak function is equivalent to reversing the compositions in the monomial quasisymmetric function expansion of the peak function. 
Suppose, for some set of permutations $\Pi$, that $R_n(\Pi)$ is symmetric. In this case, we can expand $R_n(\Pi)$ in terms of the \textit{monomial symmetric function} basis of $\Sym$. Since monomial symmetric functions expand into monomial quasisymmetric functions by summing over every rearrangement of the indexing partition, if $R_n(\Pi)$ is symmetric then for every occurrence of $M_\beta$ in $R_n(\Pi)$, $M_{\beta^r}$ also occurs. 

From this, it follows that if $R_n(\Pi)$ is symmetric, then it is equal to $R_n(\Pi^r).$
\end{proof}

\begin{remark}
Let $s_\gamma$ denote the Schur function indexed by a partition $\gamma$. There is an analogous result to Lemma \ref{reverselemma} for the pattern quasisymmetric function $\mathcal{Q}_n(\Pi)$ in \cite{HPS20}. In this case, there is not equality but rather, if $\mathcal{Q}_n(\Pi) = \sum c_\gamma s_\gamma$ then $\mathcal{Q}_n(\Pi^r) = \sum c_\gamma s_{\gamma^t}$, where $\gamma^t$ denotes the partition whose diagram is the reflection of the diagram of shape $\gamma$ about the main diagonal. The \textit{complement} $\pi^{c}$ of a permutation $\pi$ is obtained by replacing each element $x$ of the permutation with $n+1-x$, and the \textit{reverse complement} $\pi^{rc}$ of a permutation $\pi$ is the complement of $\pi^r$. In \cite{HPS20}, it is also shown that if $\mathcal{Q}_n(\Pi) = \sum c_\gamma s_\gamma$ then $\mathcal{Q}_n(\Pi^c) = \sum c_\gamma s_{\gamma^t}$, and so $\mathcal{Q}_n(\Pi^{rc}) = \sum c_\gamma s_\gamma$, where $\Pi^r = \{\pi^r:\pi\in\Pi\}$ and $\Pi^{rc} = \{\pi^{rc}:\pi\in \Pi\}$. It is easy to verify that these properties do not hold for pattern-avoiding peak functions. For example, we have $213^c = 231$ and $213^{rc} = 132$.  As we shall see shortly, $R_n(213)$ is Schur $Q$-positive for all $n$, but one can check that $R_n(132)$ and $R_n(231)$ are not even symmetric if $n=5$.

\end{remark}
Lemma \ref{reverselemma} tells us that having $R_n(\Pi) = R_n(\Pi^r)$ is a necessary condition for $R_n(\Pi)$ to be symmetric. The conclusion of Lemma \ref{reverselemma} does not necessarily hold for avoidance classes whose pattern-avoiding peak function is not symmetric. The simplest case to see this is when $\Pi = \{123,312\}$. Two permutations in $\mathrm{Av}_4(123,312)$ have peak set $\{3\}$, namely 3241 and 2143. On the other hand, in $\mathrm{Av}_4(213,321)$, the permutations 1342, 1243, and 2341 all occur, and all have peak set $\{3\}.$ Therefore $\Pi$ and $\Pi^r$ are not peak-equivalent, and $R_4(\Pi) \neq R_4(\Pi^r)$. 

One can also show that having $R_n(\Pi) = R_n(\Pi^r)$ is not a sufficient condition for $R_n(\Pi)$ to be symmetric. It is not difficult to demonstrate that the multisets $\{\!\{\Des(\pi): \pi\in \mathrm{Av}_n(132)\}\!\}$ and  $\{\!\{\Des(\pi): \pi\in \mathrm{Av}_n(231)\}\!\}$ are the same, which implies that 132 and 231 are peak-equivalent and $R_n(132) = R_n(231)$, but as noted earlier, $R_5(132)$ is not symmetric.

Defined analogously to peaks, the \textit{valleys} of a permutation $\pi$ are those indices $i$ such that $\pi_{i-1}> \pi_i < \pi_{i+1}.$ Claesson and Kitaev's survey \cite{claesson2008classification} describes a bijection between $\mathrm{Av}_n(321)$ and $\mathrm{Av}_n(132)$ based in part on work done by Rotem in \cite{rotem1975correspondence}, and another result due to Knuth \cite{knuth1969art}. In \cite{claesson2008classification}, they note that this bijection sends the peaks of a permutation in $\mathrm{Av}_n(321)$ to valleys in the image. It is simple to check that taking the reverse complement of a permutation $\pi$ sends peaks to valleys (and vice versa), and that $\{\pi^{rc}:\pi\in \mathrm{Av}_n(\Pi)\} = \mathrm{Av}_n(\Pi^{rc})$. By taking the reverse complement of permutations in $\mathrm{Av}_n(132)$, and composing this map with the bijection from \cite{claesson2008classification}, we obtain a peak set preserving bijection between $\mathrm{Av}_n(321)$ and $\mathrm{Av}_n(213)$. This, along with Lemma \ref{reverselemma}, shows that all four cases of the second row of Table \ref{mainthm} are equivalent. In particular, if we are able to show that e.g. $R_n(321)$ is symmetric, then we have $R_n(123) = R_n(213) = R_n(312) = R_n(321).$ This is the object of the following theorem. 

\begin{theorem}\label{321case}
For each $n \geq 1$,
\[
R_n(321) = \sum_{k=0}^{\lfloor n/2 \rfloor} 
\left( {n \choose k+1} - {n \choose k-1} \right) Q_{(n-k,k)} .
\]
\end{theorem}
We delay the proof of Theorem \ref{321case}, as it requires some auxiliary results. 
\begin{definition}
Given an $\SYT$, $S$, let the \textit{standard reading word} of $S$, $\mathrm{rw}(S)$, be the word formed by reading entries of rows of $S$ from left to right, starting at the top row and working down row by row to the bottom. We define a map $\Phi$, that sends a permutation $\pi\in\mathfrak{S}_n$ to a standard shifted tableau with $n$ cells, by \[\Phi(\pi) = R(\mathrm{rw}(Q(\pi))).\] 

In words, we map a permutation to a standard shifted tableau by first performing $\mathrm{RSK}$ insertion, then creating a standard reading word from the recording tableau, and sending this to an $\SShT$ via Sagan--Worley insertion. At the final step, we consider only the insertion tableau. 
\end{definition}

Although $\Phi$ is defined for any permutation, since we will use this map to establish a peak-preserving correspondence between $\mathrm{Av}(321)$ and standard shifted tableaux of height at most two, we restrict the domain of $\Phi$ to $\mathrm{Av}(321)$ throughout all that follows.

\begin{proposition}\label{peakpreservephi}
For any permutation $\pi$, 
\[\Peak(\Phi(\pi)) = \Peak(\pi).\]
\end{proposition}
\begin{proof}
We know $\Peak(\pi) = \Peak(Q(\pi))$ since $\Des(\pi) = \Des(Q(\pi))$, a well known property of $\mathrm{RSK}$ insertion. It remains to show that $\Peak(Q(\pi)) = \Peak(\Phi(\pi))$. Consider an element $p\in \Peak(Q(\pi))$. Since $p$ is a peak in the $\SYT$ $Q(\pi)$, it is strictly below $p+1$ and weakly below $p-1$. This means that $p-1$ and $p+1$ come before $p$ in $\mathrm{rw}(Q(\pi))$, and therefore, $p\in \Peak(\Phi(\pi))$ by Proposition \ref{insertionpeaks}.  

Next, suppose $q\notin \Peak(Q(\pi))$. This means that either $q$ is weakly above $q+1$, or strictly above $q-1$ and strictly below $q+1$ in $Q(\pi)$. In the first case, $q+1$ is either right of $q$ in the same row of $Q(\pi)$, or strictly below $q$, but either way, $q+1$ is right of $q$ in $\mathrm{rw}(Q(\pi))$. In the second case, $q-1$ is right of $q$ which is right of $q+1$ in $\mathrm{rw}(Q(\pi))$, and for both situations, $q\notin \Peak(\Phi(\pi))$ by Proposition \ref{insertionpeaks}. 
\end{proof}

To prove Theorem \ref{321case}, we will fix an $\SShT$ $T$ and construct the preimage $\Phi^{-1}(T)$. We will show that the preimage of $T$ consists of all permutations whose recording tableau under $\mathrm{RSK}$ insertion is one of (up to) two distinct $\SYT$, one of which, $S$, is obtained in a very natural manner from $T$, and the other, $S'$, is found by \textit{shifting} one entry from the bottom row of $S$ to the top row. We use indices to refer to specific entries within tableaux, i.e. $T_{i,j}$ is the entry in the cell of a tableau $T$ that lies in the $i^{th}$ row from the bottom, and $j^{th}$ column from the left. We now describe the tableaux $S$ and $S'$, and the shifting process. 

\begin{definition}
Fix some standard shifted tableau $T\in \SShT(n-k,k)$, where $n\geqslant 2k+1$. Let $S$ denote the tableau obtained by moving the entire second row of $T$ one cell to the left. If $n > 2k+1$, then let $S'$ be the tableau created in the following manner. 
Let $s$ denote the largest entry of $S$ in row 1 such that the entry in row 2, two columns to the left, is smaller than $s$. That is, $s = S_{1,j+2}$ is the largest entry such that $S_{1,j+2} > S_{2,j}$, for some $j$. If no such entry exists, let $s=2$. Then $S'$ is the tableau obtained by removing $s$ from row 1 of $S$ and placing $s$ into row 2 as follows. When removing $s$ from row 1, we move every entry from row 1 that is larger than $s$ to the left by one, and delete the empty cell left at the rightmost end of row 1. When placing $s$ into row 2, we create a cell in the rightmost position of row 2, shift every entry larger than $s$ one cell to the right, and place $s$ into the empty cell.

\end{definition}

Note that $S$ is clearly an $\SYT$. If $n>2k+1$ and we can create $S'$, then $S'$ is also an $\SYT$, which we will prove.

\begin{lemma}
For $T\in \SShT(n-k,k)$, if $n>2k+1$, $S'$ is an $\SYT$. 
\end{lemma}
\begin{proof}
Note that in $S'$, $s$ must be placed one column to the left of where it was originally situated in the bottom row of $S$. If $s>2$, this is because the entry in the second row of $S$ and two columns to the left of $s$ is smaller than $s$, but the entry one column to the left (if this cell is nonempty) was above $s$ in $T$, so is larger than $s$. If $s=2$, then $s$ is smaller than every entry in the second row of $S$. By construction, the rows and columns of $S'$ are increasing, so it is an $\SYT$. 
\end{proof}

\begin{example}\label{SS'Ex}
Let \[T = \begin{array}{c@{\hskip \cellsize}c@{\hskip \cellsize}c@{\hskip \cellsize}c@{\hskip \cellsize}c@{\hskip \cellsize}c}\tableau{&4&5&8\\1&2&3&6&7&9}\end{array}\!\!\!\!\!\!. \]

Then  \[S = \begin{array}{c@{\hskip \cellsize}c@{\hskip \cellsize}c@{\hskip \cellsize}c@{\hskip \cellsize}c@{\hskip \cellsize}c}\tableau{4&5&8\\1&2&3&6&7&9}\end{array}\!\!\!\!\!\!. \]
\vspace{0.2 cm}

The largest entry in the bottom row of $S$ that is larger than the entry of the cell above and two spaces to the left is $6$, so here $s=6$. We shift $s$ into the second row to produce 

 \[S' = \begin{array}{c@{\hskip \cellsize}c@{\hskip \cellsize}c@{\hskip \cellsize}c@{\hskip \cellsize}c@{\hskip \cellsize}c}\tableau{4&5&6&8\\1&2&3&7&9}\end{array}\!\!\!\!\!\!. \]
 \vspace{0.2 cm}
 
The reading words of these tableaux are $\mathrm{rw}(S) =458123679 $ and $\mathrm{rw}(S') = 456812379$. One may check that performing Sagan--Worley insertion on these reading words gives insertion tableau $T$. 

\end{example}

\begin{lemma}\label{SS'unique}
Given some $T\in \SShT(n-k,k)$, if $n>2k+1$ then \[\Phi^{-1}(T)=\{\pi\in \mathrm{Av}_n(321): Q(\pi)\in \{S,S'\}\},\] and if $n=2k+1$, then \[\Phi^{-1}(T) = \{\pi\in \mathrm{Av}_n(321): Q(\pi) = S\}.\]
\end{lemma}

\begin{proof}
To see that $\mathrm{rw}(S)$ is sent to $T$ by Sagan--Worley insertion, note that $\mathrm{rw}(S)$ is two increasing words $uw$, where $u$ is the top row of $T$, and $w$ the bottom row of $T$. It follows that $|w| > |u|$ and $w=12w_3\cdots w_{n-k}$. Performing Sagan--Worley insertion on $u$ will give a single-row insertion tableau whose entries are precisely $u$. Inserting 1 will then column bump everything to the right. Then, for all $i\in [k]$, since $u_i < w_{i+1}$, $u_i$ will be bumped to the second row of the insertion tableau by $w_{i+1}$. Since the entries of $w$ are increasing, none of these will be bumped, and so $R(\mathrm{rw}(S)) = R(uw) = T$.

Next we consider $\mathrm{rw}(S')$. Again we write this as $u'w'$. Now however, $u' = u_1\ldots u_jsu_{j+1}\ldots u_k$, and $w' = 1w_2\ldots w_{j+1}w_{j+3}\ldots w_{n-k}$ for some $j$, i.e. $s= w_{j+2}$. 

It is useful to visualize $T$ with this labeling of entries:

\[T = \begin{array}{c@{\hskip \cellsize}c@{\hskip \cellsize}c@{\hskip \cellsize}c@{\hskip \cellsize}c@{\hskip \cellsize}c}\Ttableau{&u_1&\ldots&u_j&u_{j+1}&\ldots&\ldots &u_k\\1&2&\ldots&w_{j+1}&s&w_{j+3}&\ldots&w_{k+1}&\ldots&w_{n-k}}\end{array}\!\!\!\!\!.\]

As before, $u'$ is an increasing word, so each consecutive element will insert at the end of a single row, meaning the insertion tableau of $u'$ is simply \[R(u') = \begin{array}{c@{\hskip \cellsize}c@{\hskip \cellsize}c@{\hskip \cellsize}c@{\hskip \cellsize}c@{\hskip \cellsize}c}\Ttableau{u_1&\ldots&u_j&s&u_{j+1}&\ldots &u_k}\end{array}\!\!\!\!\!.\]

Note that $w'$ contains the entry 1, and precisely $j$ other elements that are smaller than $s$. Based on their position in $T$, we know that $w_{i+1} < u_i$ for all $i\in[j]$, so $w_2$ will bump $u_1$, $w_3$ will bump $u_2$, and so on, until the insertion tableau has the following form:
\[R(u'w_1\ldots w_{j+1}) = \begin{array}{c@{\hskip \cellsize}c@{\hskip \cellsize}c@{\hskip \cellsize}c@{\hskip \cellsize}c@{\hskip \cellsize}c}\Ttableau{&u_1&\ldots&u_j\\ 1&2&\ldots&w_{j+1}&s&u_{j+1}&\ldots &u_k}\end{array}\!\!\!\!\!.\]

Now $s = w_{j+2}$ is by definition smaller than $w_i$ for $i\in[j+3,n-k]$, so $s$ will not be bumped while inserting the remaining elements of $\mathrm{rw}(S)$. Also, $s$ was chosen as the maximal entry $S_{1,j+2}$ that is larger than $S_{2,j}$, meaning $u_i < w_{i+2}$ for $i\in[j+1,k]$. It follows that every element of $u'$ that is larger than $s$ will be bumped, and $R(u'w') = T$, as desired. 

Since $\mathrm{RSK}$ insertion maps permutations in $\mathrm{Av}_n(321)$ to all pairs of $\SYT$ of height at most 2, it remains to show that no other $\SYT$ with at most two rows has a reading word which is mapped to $T$ under Sagan--Worley insertion. Suppose $P\in \SYT(a,b)$ such that $R(\mathrm{rw}(P)) = T$. As before, $\mathrm{rw}(P) = uw$ for two increasing subwords $u$ and $w$, where $|u| = b$, $|w|=a$, and $a\geqslant b$. For every element $u_i$ in $u$ there is an element $w_i$ in $w$ that is smaller than $u_i$. As $w_1=1$, which will column bump $u_1$ during insertion, either every element in $u$ is row bumped, or every element but one is row bumped, as may occur if $w_{j+1} > u_j$ for some $j$, or if $b=a$. Hence, $b\in[k,k+1]$. It follows that $u$ contains every entry from the top row of $T$, and at most one other element. If $b = k$ then we must have $P=S$, so assume $b=k+1$, and $u$ contains some element $p\neq s$ such that $p$ is an entry in the bottom row of $T$. We have two cases to consider.

Case 1: Suppose $p<s$. Then, since $p$ must be an entry of the first row of $T$, moving the second row of $T$ left by one and then shifting $p$ into this row to create $P$ will either place $p$ above $s$, or place $T_{2,j}$ above $s$, in either instance ensuring that $P$ is not an $\SYT$.

Case 2: Suppose $p>s$. Consider forming $P$ by shifting $p$ up from the bottom row of $S$.  Recall $s$ was the largest entry in the bottom row of $S$ that is larger than the entry above and two cells to the left. This means that in $S$, $p$ is smaller than the entry above and two cells to the left, if this entry exists. When $p$ is placed into the top row of $S$, its position will be at least two cells left of where it was in the bottom row of $S$. This implies that there are more entries in the bottom row of $P$ that are smaller than $p$, than there are entries in the top row of $P$ that are smaller than $p$. This means $p$ will be bumped to the second row of $R(\mathrm{rw}(P))$, and so $R(\mathrm{rw}(P))\neq T$. 

Note that when $n=2k+1$, the $\SYT$ whose reading words map to $T$ under Sagan--Worley insertion must have $k$ entries in the top row. Since these must be the entries that form the second row of $T$, the only $\SYT$ of shape $(k+1,k)$ whose reading word maps to $T$ is $S$.
\end{proof}

Now that we have established what the preimage under $\Phi$ of an arbitrary element from $\SShT(n-k,k)$ is, we are ready to prove Theorem \ref{321case}. 

\begin{proof}[Proof of Theorem \ref{321case}]

Let $T\in \SShT(n-k,k)$. 

Suppose $n > 2k+1$. We wish to count the permutations in $\mathrm{Av}_n(321)$ that are sent to $S$ and $S'$ under $\mathrm{RSK}$ insertion. Since $\mathrm{RSK}$ insertion maps permutations in $\mathrm{Av}_n(321)$ to all pairs of $\SYT$ of height at most 2, this amounts to counting the tableaux in $\SYT(n-k,k)$ and $\SYT(n-k-1,k+1)$. In other words, there are $|\SYT(n-k,k)|$ permutations in $\mathrm{Av}_n(321)$ with recording tableau $S$, and $|\SYT(n-k-1,k+1)|$ with recording tableau $S'$. One can easily show that $|\SYT(a,b)| = {a+b\choose b} - {a+b\choose b-1}$. 

Given a Schur $Q$-function $Q_{(n-k,k)}$, by \eqref{SchurQdefn}, each peak function in the peak expansion of $Q_{(n-k,k)}$ is associated to a specific tableau $T\in \SShT(n-k,k)$. The map $\Phi$ allows us to associate permutations in $\mathrm{Av}_n(321)$ with these tableaux, so we wish to count the number of permutations that are mapped to a specific tableau:

\begin{align*}
    |\{\pi\in \mathrm{Av}_n(321):\Phi(\pi)=T\}| &=|\mathrm{SYT}(n-k-1,k+1)|+|\mathrm{SYT}(n-k,k)|  \\[6 pt] &= {n\choose k+1}-{n\choose k}+{n\choose k} - {n\choose k-1} \\[6 pt] &={n\choose k+1} - {n\choose k-1}.
\end{align*}

If $n=2k+1$ the enumeration is still correct since, in this case, ${n\choose k} = {n\choose k+1}$. This establishes that $|\Phi^{-1}(T)| ={n\choose k+1} - {n\choose k-1}$ for any $T\in \SShT(n-k,k)$. 

Let $\lambda = (n-k,k)$, and let $A(\lambda) = \bigcup_{T\in \SShT(\lambda)} \Phi^{-1}(T)$. Summing peak functions indexed by permutations in $A(\lambda)$, since $\Phi$ preserves peak sets by Proposition \ref{peakpreservephi}, we have 

\begin{align*}\sum_{\pi\in A(\lambda)} K_{\Peak(\pi)} &= \sum_{T\in \SShT(\lambda)} \left( {n\choose k+1} - {n\choose k-1} \right) K_{\Peak(T)}\\[6 pt] &= \left({n\choose k+1} - {n\choose k-1}\right) Q_\lambda.\end{align*} The result follows from summing over $k$ such that $n \geqslant 2k+1$.
\end{proof}

\subsection{The cases with two patterns}
\begin{proposition}\label{Vperms}
\[R_n(132,231) = 2^{n-1}Q_{(n)}.\]
\end{proposition}
\begin{proof}
The class $\mathrm{Av}(132,231)$ is precisely the permutations with no peaks. These permutations consist of a subsequence of decreasing elements ending with 1, followed by a subsequence of increasing elements, giving them a distinctive $V$-shape (see Figure \ref{Vshape}). Given $\pi\in \mathrm{Av}_n(132,231)$, let us call the decreasing part $\pi_L$ and the increasing part $\pi_R$. To see that $|\mathrm{Av}_n(132,231)| = 2^{n-1}$, note that once the elements in the decreasing subsequence $\pi_L$ are chosen, the permutation is completely specified, as every remaining element must be placed in increasing order. The smallest element, 1, is defined to be in $\pi_L$, and then there are $n-1$ other elements which may or may not be in $\pi_L$. Hence, there are $n-1$ choices to make and so $|\mathrm{Av}_n(132,231)| = 2^{n-1}.$ The result follows, as \[R_n(132,231) = 2^{n-1}K_{\emptyset,n}= 2^{n-1}Q_{(n)}.\]

\vspace{-0.5 cm} 
\end{proof}

\begin{theorem}\label{132312}
    \[R_n(132,312) = \sum_{k=0}^{\lfloor \frac{n}{2} \rfloor} (n-2k)Q_{(n-k,k)}.\]
\end{theorem}

The proof of this theorem involves performing Sagan--Worley insertion on the inverses of permutations in $\mathrm{Av}_n(132,312)$. This is the class $\mathrm{Av}(132,231)$. We will show that these permutations are mapped to shifted tableaux of height at most 2, and by assigning certain labels to entries in a given tableau, we can count how many permutations are sent to each tableau. Before proving the theorem we will describe the labeling. 

Given some $T\in \SShT(n-k,k)$ for $n\geqslant2k+1$, we label the entries in the top row $a_1, a_2,\ldots, a_k$ in increasing order. Then, let $b_k$ be the largest entry from the bottom row of $T$ that is smaller than $a_k$. 
Assign the labels $b_{k-1},\ldots,b_2,b_1$ in decreasing order, where for each $i\in[k-1]$, $b_i$ be the largest entry from the bottom row of $T$ that is smaller than both $a_i$ and $b_{i+1}$. Label the remaining $n-2k$ entries in increasing order as $c_1,\ldots,c_{n-2k}$.

We continue to use the notation $\pi_L$ and $\pi_R$ to identify the decreasing and increasing parts of a permutation $\pi$ in the class $\mathrm{Av}(132,231)$. We will show in the proof of Theorem \ref{132312} that the permutations that are sent to $T$ by Sagan--Worley insertion are those for which $a_i\in \pi_L$ and $b_i\in \pi_R$ for all $i\in[k]$, and $c_1,c_2,\ldots,c_j$ in $\pi_L$ and $c_{j+1},c_{j+2},\ldots,c_{n-2k}\in \pi_R$ for some $1\leqslant j \leqslant n-2k$. Clearly given some $n$ and $k$, there are $n-2k$ such permutations for each filling of $\SShT(n-k,k)$. 

\begin{example}\label{2patternexample}
For the shifted tableau \[T = \tableau{&4&7&8\\1&2&3&5&6&9} \] we have $a_1=4,a_2=7, $ and $a_3=8$. The largest entry from the first row that is smaller than 8 is 6, so this is $b_3$. The largest entry from the bottom row that is smaller than $a_2$ and $b_3$ is 5, so this is $b_2$, and the largest entry smaller than $a_1$ and $b_2$ is 3, so $b_1=3$. The remaining entries are assigned labels in increasing order, $c_1 = 1$, $c_2=2$ and $c_3=9$. 

The three permutations associated with this labeling are then $874123569,$ $874213569,$ and $987421356$ (see Figure \ref{Vshape}). 

\begin{figure}[h]
\begin{center}
\begin{tikzpicture}[scale=0.5]
    
    \foreach \x in {0,11,22} {
        \draw[gray, thin] (\x,0) grid (\x+10,10);
        \draw[black] (\x,0) rectangle (\x+10,10);
    }

    \foreach \i/\y in {1/8, 2/7, 3/4, 4/1, 5/2, 6/3, 7/5, 8/6, 9/9} {
        \fill (\i, \y) circle (6pt);
    }

    \foreach \i/\y in {12/8, 13/7, 14/4, 15/2, 16/1, 17/3, 18/5, 19/6, 20/9} {
        \fill (\i, \y) circle (6pt);
    }

    \foreach \i/\y in {23/9, 24/8, 25/7, 26/4, 27/2, 28/1, 29/3, 30/5, 31/6} {
        \fill (\i, \y) circle (6pt);
    }

    \foreach \i/\y in {4/1, 5/2, 9/9, 16/1, 15/2, 20/9, 28/1, 27/2, 23/9} {
        \fill[red] (\i, \y) circle (6pt);
    }
    \node at (5.2, -1.5) {$874123569$};
    \node at (16, -1.5) {$874213569$};
    \node at (27.2, -1.5) {$987421356$};

\end{tikzpicture}
\end{center}
\caption{The permutations in $\mathrm{Av}(132,231)$ with insertion tableau $T$ under Sagan--Worley insertion. The entries $c_i$, which may be in the decreasing or increasing subsequences, are colored red. }\label{Vshape}
\end{figure}
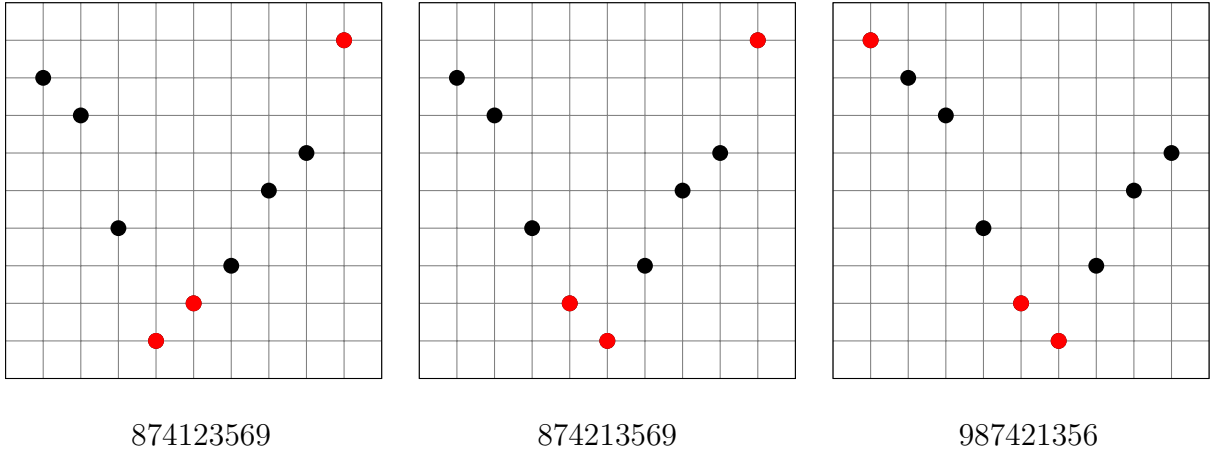
\end{example}

\begin{proof}[Proof of Theorem \ref{132312}]
First, we claim that for $\pi\in \mathrm{Av}_n(132,231)$, Sagan--Worley insertion sends $\pi$ to a tableau with height at most 2. To see this, note that the elements of $\pi_L$ will all insert at the beginning of the first row via column bumping. After inserting the rightmost element of $\pi_L$, 1, we have a single-row tableau. Row bumping may occur when inserting elements from $\pi_R$, but at each insertion, since $\pi_R$ is increasing, any entry that is row bumped will be larger than the last entry that was row bumped. This means that row bumping can only occur once per insertion and so $R(\pi)$ cannot have height more than 2. 

Fix some filling $T\in \SShT(n-k,k)$, for some $n\geqslant 2k+1$. Consider the permutations associated with this tableau filling via the labeling of entries described above.

When we perform Sagan--Worley insertion on these permutations the only elements that may row bump an entry are those from the increasing subsequence, $\pi_R$. This subsequence consists of the $b$ elements, and a (possibly empty) subset of the largest $c$ elements. We claim that when inserting elements of $\pi_R$, $a_i$ will be bumped, for all $i\in [k]$. Note that $b_i+1$ is either some $a_h$ for $h\leqslant i$, or possibly $b_{i+1}$. In particular, $b_i+1\neq c_j$, since we assign the $b$ labels before the $c$ labels and the $b_i$ label is given to the largest entry of $T$ that is less than $a_i$ and has not yet had a label assigned. 

When we insert $b_1$, $b_{2}$ has not yet been inserted, but $a_1$ has, so $a_1$ will be bumped by $b_1$ unless $a_1$ has already been bumped. Inducting on $i$, it follows that every $a_i$ entry will have been bumped to the second row of the insertion tableau following the insertion of $b_i$. Also, none of the entries in the second row are column bumped back to the first row since, as noted previously, every entry that is row bumped to the second row is the largest entry bumped to that row thus far. 

We also must show that no $b$ or $c$ entries are bumped. Clearly no $b_i$ can be bumped, as these are all in $\pi_R$, and every element that is inserted later than $b_i$ is larger than $b_i$. The $c$ elements from $\pi_R$ cannot be bumped for the same reason. To see that no $c_j\in\pi_L$ can be bumped, note that if $c_l\in \pi_R$ then $c_l>c_j$, so $c_l$ will not bump $c_j$. Suppose some $c_j$ was bumped by $b_i<c_j$. Then, in particular, $b_i$ did not bump $a_i$, meaning that some $a$ entry has been bumped by some $c_h\in \pi_R$ such that $c_h < b_i$. However, this implies that $c_j\in \pi_R$, since $c_h < b_i < c_j$. This contradicts our assumption that $b_i$ bumped $c_j$. 

Therefore, we have shown that Sagan--Worley insertion maps $\pi$ to $T$. To show that no permutations other than those labeled as described will map to $T$, we will show that this labeling accounts for every permutation in $\mathrm{Av}_n(132,231)$. We know from the proof of Proposition \ref{Vperms} that $|{\mathrm Av}_n(132,231)| = 2^{n-1}$, so it remains to show that \begin{equation}\label{eqnopeaks}
    \sum_{k=0}^{\left \lfloor \frac{n}{2}\right \rfloor}(n-2k)|\SShT(n-k,k)| = 2^{n-1}.
\end{equation}

One can show using a standard path counting and reflection principle argument that $|\SShT(n-k,k)| = {n-1\choose k}-{n-1 \choose k-1}.$ This allows us to write \eqref{eqnopeaks} as a telescoping sum, which we expand as follows: 
\[
    \sum_{k=0}^{\left \lfloor \frac{n}{2}\right \rfloor}(n-2k)|\SShT(n-k,k)| =\sum_{k=0}^{\left \lfloor \frac{n}{2}\right \rfloor}(n-2k)\left({n-1\choose k}-{n-1 \choose k-1} \right) \]
    
    \[ = n{n-1\choose 0} +(n-2)\left({n-1\choose 1}-{n-1 \choose 0} \right) +\cdots + \left(n - 2\left\lfloor \frac{n}{2}\right\rfloor\right)\left({n-1\choose \left\lfloor \frac{n}{2}\right\rfloor}-{n-1 \choose \left\lfloor \frac{n}{2}\right\rfloor-1} \right) 
\]

\[=2{n-1\choose 0} + 2{n-1 \choose 1} +2{n-1\choose 2}+\cdots+2{n-1 \choose \left\lfloor \frac{n}{2}\right\rfloor-1}+\left(n-2\left\lfloor \frac{n}{2}\right\rfloor\right){n-1 \choose \left\lfloor \frac{n}{2}\right\rfloor}\]

\[=\sum_{k=0}^{n-1}{n-1\choose k} = 2^{n-1}.\]

The penultimate equality is because $2{n-1\choose k} = {n-1\choose k}+{n-1 \choose n-1-k}$ and 

    \[n-2\left\lfloor \frac{n}{2}\right\rfloor = \begin{cases}
        1 & \text{ if $n-1$ is even} \\ 0 &\text{ if $n-1$ is odd.}
    \end{cases} \]

Fix some strict partition $\lambda = (n-k,k)$. For any $T\in \SShT(\lambda)$, there are $n-2k$ entries that we label as $c_i$, for $i\in [n-2k]$. As we have seen, these correspond to $n-2k$ permutations in $\mathrm{Av}_n(132,231)$ which have insertion tableau $T$ under Sagan--Worley insertion. Let $A(T)$ denote the set $\{\pi\in \mathrm{Av}_n(132,312): R(\pi^{-1})=T\}$. If we take $A(\lambda) = \bigcup_{T\in \SShT(\lambda)}A(T)$, then summing peak functions indexed by the peak sets of permutations in $A(\lambda)$ gives \[\sum_{\pi\in A(\lambda)}K_{\Peak(\pi)} = \sum_{T\in \SShT(\lambda)}(n-2k)K_{\Peak(T)} = (n-2k)Q_\lambda.\]
The result follows from summing over all strict partitions of length at most 2.
\end{proof}

Theorem 1.2 in \cite{HPS20} establishes that permutations in $\mathrm{Av}_n(213,231)$, $\mathrm{Av}_n(213,132)$, and $\mathrm{Av}_n(132,123)$ share the same multiset of descent sets. From that result, the following is immediate:

\begin{proposition}\label{pekequiv2pattersn}
    The sets of patterns $\{213,231\}$, $\{213,132\}$, and $\{132,123\}$ are peak-equivalent.
\end{proposition}

Proposition \ref{pekequiv2pattersn}, Theorem \ref{132312}, and Lemma \ref{reverselemma} together imply the third line of Table 1.

The next result is not part of Theorem \ref{mainthm}, but it follows from Theorem \ref{132312} and Theorem \ref{321case}. 
\begin{corollary}\label{includedeltaj}
For a decreasing permutation $\delta_j$ of size $j$, the pattern-avoiding peak function associated with $\mathrm{Av}(132,312,\delta_j)$ is given by \[R_n(132,312,\delta_j) = \sum_{k=0}^{j-2} (j-1-k)Q_{(n-k,k)}.\]
\end{corollary}

\begin{proof}

Let $T\in \SShT(n-k,k)$ for some $n,k$. The permutations within $\mathrm{Av}_n(\{132,312,\delta_j\}^{-1}) = \mathrm{Av}_n(132,231,\delta_j)$ that insert to $T$ under Sagan--Worley insertion are constructed as in Proposition \ref{132312}, but with at most $j-1$ decreasing elements in the left arm of the V-shape. Since $1$ must be included in this, there can be at most $j-2$ other elements, namely $a_i$ for $i\in [k]$, and up to $j-1-k$ other elements from the $c_i$ elements.
\end{proof}

\subsection{The remaining cases.}

We will now adapt a proof from \cite{HPS20} to prove the cases when $\Pi$ is the shuffle product of two sets of permutations. For two elementwise-disjoint sequences $\pi$ and $\sigma$ of length $m$ and $n$ respectively, the \textit{shuffle product} $\pi\shuffle \sigma$ is all sequences of length $m+n$ that have both $\pi$ and $\sigma$ as subsequences. For permutations, we add $m$ to every element of $\sigma$ in order to ensure that the product is a set of permutations. For sets of permutations $\Pi$ and $\Pi'$, we let \[\Pi \shuffle \Pi' = \{\pi\shuffle \sigma: \pi\in \Pi, \sigma\in \Pi'\}.\]

\begin{theorem}\label{shuffleformula}
For any sets of nonempty permutations $\Pi$ and $\Pi'$, and any $n\geq 0$, the following holds: 

    \begin{equation} \label{shuffprodeq}
    R_n(\Pi\shuffle \Pi') = R_n(\Pi') +\sum_{k=0}^{n-1} R_k(\Pi)[K_{\emptyset,1}R_{n-k-1}(\Pi')-R_{n-k}(\Pi')]
    \end{equation}

\end{theorem}

The proof of Theorem \ref{shuffleformula} differs from the proof of Theorem 3.1 in \cite{HPS20} in one detail. Rather than reproduce the proof in full, we outline the point of difference between our proof and theirs. The proof in \cite{HPS20} uses the homomorphism $\Phi$ from the Malvenuto--Reutenauer algebra (MR) of formal $\mathbb{Q}-$linear combinations of permutations to $\QSym$, where $\Phi(\pi) = F_{\Des(\pi)}.$ We replace this step of the proof with a map $\varrho: \mathrm{MR}\mapsto \QSym$ given by $\varrho(\pi) = K_{\Peak(\pi)}$.

The product in $\mathrm{MR}$ is the shuffle product of two permutations. The product of two peak functions $K_{S_1,n_1}$ and $K_{S_2,n_2}$ is obtained by taking two permutations $\pi\in \mathfrak{S}_{n_1}$ and $\sigma\in \mathfrak{S}_{n_2}$, such that $\Peak(\pi) = S_1$ and $\Peak(\sigma)=S_2$. Then, \[K_{S_1,n_1}K_{S_2,n_2} = \sum_{\rho\in \pi\shuffle \sigma}K_{\Peak(\rho),n_1+n_2}.\] From this it follows that the map $\varrho$ is a homomorphism. The rest of the proof of Theorem \ref{shuffleformula} is identical to the proof of Theorem 3.1 in \cite{HPS20}.

\begin{theorem}\label{shufprodthm}
If $n\geqslant 2$ and \[\Pi \in \big\{\{123,132,312\},\{123,213,231\},\{321,312,132\},\{213,231,321\}\big\},\]\ then $R_n(\Pi) = 2Q_{(n)}+Q_{(n-1,1)}$. 
\end{theorem}
\begin{proof}
Each of these sets of patterns is the shuffle product of a pattern of size 2 with a pattern of size 1, for example $R_n(123,132,312)=R_n(12\shuffle 1).$ Let $\Pi'=1$ and note that since the only permutation that avoids $\Pi'$ is the empty permutation, and $K_{\emptyset,0}=1$ by definition, we have $R_k(1) = \delta_{k,0}$, the Kronecker delta. If $n\geq 2$, then \eqref{shuffprodeq} becomes \[R_n(12\shuffle 1) = R_{n-1}(12)K_{\emptyset,1} = K_{\emptyset,n-1}K_{\emptyset,1}\]
This product of these peak functions is evaluated by placing $n$ into every possible position of the permutation $\iota_{n-1}$ and considering the peak sets. Placing $n$ at the first or last position creates two permutations with peak set $\emptyset$, and every other placement creates a peak set $\{m\}$ for $m\in [2,n-1]$. These correspond to fillings of $\SYT(n-1,1)$. The result follows. 
\end{proof}

For the cases in line 5 of Table \ref{mainthm} that are not shuffles, we will construct peak-preserving bijections between two of these and $\{123,132,312\}$, which was shown to have a symmetric pattern-avoiding peak function in Theorem \ref{shufprodthm}. We outline these bijections now. 
\begin{proposition}\label{peakequivshuff}
    The sets of patterns $\{123, 132, 312\},~\{132, 213, 321\},$ and $\{132, 213, 312\}$ are all peak-equivalent. 
\end{proposition}
\begin{proof}
If $\pi\in \mathrm{Av}_n(123,132,312)$, then $\pi$ has at most one ascent, and all elements of $\pi$ to the left of $n$ are larger than all elements to the right of $n$. This means that $\pi$ has the form \[\pi = (\delta_{n-k-1}+k) n \delta_k \] for some $k\in [n-1].$
One can similarly write down specifications for the form of permutations in the other two classes. These make it clear that to specify a particular permutation within these classes, one only needs to know the location of the largest element, $n$. From there, it is simple to see that the map \[\pi \mapsto (\iota_{n-k-1}+k)n\delta_k\] is a peak-preserving bijection between $\mathrm{Av}_n(123,132,312)$ and $\mathrm{Av}_n(132,213,312)$, and the map 
\[\pi \mapsto (\iota_{n-k-1}+k)n\iota_k\] is a peak-preserving bijection between $\mathrm{Av}_n(123,132,312)$ and $\mathrm{Av}_n(132,213,321)$.
\end{proof}
The bijections from Proposition \ref{peakequivshuff} combined with Theorem \ref{shufprodthm} establish that \[R_n(132,213,321) = R_n(132,213,312) = 2Q_{(n)} + Q_{(n-1,1)}.\]
For the remaining two cases, $\{123,231,312\}$ and $\{213,231,312\}$, note that these are the reversals of $\{132,213,321\}$ and $\{132,213,312\}$, so by Lemma \ref{reverselemma}, Theorem \ref{shufprodthm}, and Proposition \ref{peakequivshuff}, we have shown line 5 of Table \ref{mainthm} to be true.

\begin{proposition}\label{nQncase}
For $\Pi \in \big\{\{123,132,231\},\{132,213,231\}\big\}$, $R_n(\Pi) = nQ_{(n)}$. 
\end{proposition}
\begin{proof}

Both of these sets of patterns include 132 and 231, and permutations that avoid 132 and 231 have no peaks. If $\pi\in  \mathrm{Av}_n(123,132,231)$, then $\pi$ cannot have two consecutive ascents, meaning it must have $n-2$ descents, followed by either an ascent or another descent. The former case gives us $n-1$ permutations, and the latter, just $\delta_n$. Hence, $|\mathrm{Av}_n(123,132,231)| = n$ and $R_n(123,132,231) = nQ_{(n)}$. 

If $\pi\in \mathrm{Av}_n(132,213,231)$, then $\pi$ consists of descents followed by ascents, where every element of $\pi$ to the right of 1 is smaller than every element of $\pi$ to the left of 1. Hence $\pi$ is completely determined by the position of 1, giving $n$ permutations in this class, all of which have peak set $\emptyset$. 
\end{proof}
Combined with Lemma \ref{reverselemma}, Proposition \ref{nQncase} proves the $6^{th}$ line of Table \ref{mainthm}.

\begin{proposition}
For $\Pi \in \big\{ \{123,132,213,231\},\{123,132,231,312\},\{132,213,231,312\}\big\}$, $R_n(\Pi) = 2Q_{(n)}$.
\end{proposition}
\begin{proof}
For $\Pi = \{123,132,213,231\}$, allowable patterns are $312,321$. Therefore, $\mathrm{Av}_n(\Pi) = \{\delta_n,n(n-1)\ldots 312\}$. Both of these have empty peak set. The other cases are similar, and have been omitted. 
\end{proof}
Combining the above proposition with Lemma \ref{reverselemma} we get the $7^{th}$ line in Table \ref{mainthm}. 
\begin{proposition}
For $\Pi = \{123,132,213,231,312\}$, $R_n(\Pi) = Q_{(n)}$.  
\end{proposition}
\begin{proof}
For any $n\geq 3$, $\mathrm{Av}_n(\Pi) = \delta_n$, which has peak set $\emptyset$. 
\end{proof}
Combining the previous proposition with Lemma \ref{reverselemma} gives the $8^{th}$ and final line of Table \ref{mainthm}, completing the proof of Theorem \ref{MainTheoremText}.

\section{Further directions}
A common question in algebraic combinatorics involves investigating under which circumstances some class of symmetric or quasisymmetric functions is Schur-positive. Although the problem of Schur $Q$-positivity which we have considered in this work is less studied, it is arguably no less interesting from a purely combinatorial perspective. Here we have only considered avoidance classes which avoid subsets of $\mathfrak{S}_3$. We make the following conjecture:
\begin{conjecture}\label{monotoneschurqpos}
    For any increasing permutation $\iota_k$, the pattern-avoiding peak function $R_{n}(\iota_k)$ is symmetric and Schur $Q$-positive for all $n\geq 0$. 
\end{conjecture}
We know from Theorem 1.2 of \cite{zhou2024equidistribution} that peak sets are equidistributed between $\mathrm{Av}_n(\iota_k)$ and $\mathrm{Av}_n(\delta_k)$, so Conjecture \ref{monotoneschurqpos} at least passes the first hurdle of satisfying Lemma \ref{reverselemma}. It seems likely that some generalization of Theorem \ref{321case} may work to establish this conjecture, although exactly which $\SYT$ map to which $\SShT$ is not as obvious. The expansions of $R_n(\iota_k)$ into Schur $Q$-functions for $k\in\{4,5\}$ and $n\in[9]$ are given in an appendix following this section.

In Lemma \ref{reverselemma} we establish a necessary, but not sufficient, condition to determine when $\Pi$ is symmetric. A natural question to ask is the following:

\begin{question}
Is there any property of $\Pi$ that provides a sufficient condition to determine if $R_n(\Pi)$ is symmetric and/or Schur $Q$-positive?
\end{question}

It may also be interesting to consider the question of Schur-positivity or Schur $Q$-positivity for other set-valued statistics. We bring attention to one such open problem, the Schur-positivity of \textit{big descents}, which Elizalde, Rivera, and Zhuang  \cite{elizalde2024counting} conjecture to hold in $\mathrm{Av}_n(\iota_k)$ for all $k$. 

\section*{Acknowledgements} The author would like to thank their supervisor Dominic Searles for offering advice regarding exposition in this article.

\section*{Appendix: Data for Conjecture \ref{monotoneschurqpos}}

\vspace{-0.2 cm}
\begin{table}[H]
\centering

\!\!\!\!\!\!\!\!\begin{tabular}{|p{0.25cm}|l|}
\hline
  \(\mathbf{n}\) & \(\mathbf{R_n(1234)}\) \\ \hline
  1 & $Q_{(1)}$ \\[6pt]
  2 & $2Q_{(2)}$ \\[6pt]
  3 & $4Q_{(3)}+2Q_{(2,1)}$ \\[6pt]
  4 & $7Q_{(4)}+8Q_{(3,1)}$\\[6pt]
  5 & $11Q_{(5)}+20Q_{(4,1)}+16Q_{(3,2)}$\\[6pt]
  6 & $16Q_{(6)}+40Q_{(5,1)}+61Q_{(4,2)}+15Q_{(3,2,1)}$\\[6pt]
  7 & $22Q_{(7)}+70Q_{(6,1)}+155Q_{(5,2)}+91Q_{(4,3)} +77Q_{(4,2,1)}$ \\[6pt]
  8 & $29Q_{(8)}+112Q_{(7,1)}+323Q_{(6,2)}+344Q_{(5,3)} +232Q_{(5,2,1)}+168Q_{(4,3,1)}$\\[6pt]
  9 & $37Q_{(9)}+168Q_{(8,1)}+595Q_{(7,2)}+891Q_{(6,3)} +456Q_{(5,4)}+555Q_{(6,2,1)}+744Q_{(5,3,1)}+168Q_{(4,3,2)}$\\
  
  \hline
\end{tabular}
\end{table}
\begin{table}[h]
\centering

\!\!\!\!\!\!\!\!\!\!\!\!\!\!\!\!\begin{tabular}{|p{0.25cm}|l|}
\hline
  \(\mathbf{n}\) & \(\mathbf{R_n(12345)}\) \\ \hline
  1 & $Q_{(1)}$ \\[6pt]
  2 & $2Q_{(2)}$ \\[6pt]
  3 & $4Q_{(3)}+Q_{(2,1)}$ \\[6pt]
  4 & $8Q_{(4)}+8Q_{(3,1)}$\\[6pt]
  5 & $15Q_{(5)}+24Q_{(4,1)}+16Q_{(3,2)}$\\[6pt]
  6 & $26Q_{(6)}+59Q_{(5,1)}+80Q_{(4,2)}+16Q_{(3,2,1)}$\\[6pt]
  7 & $42Q_{(7)}+125Q_{(6,1)}+259Q_{(5,2)}+160Q_{(4,3)} +112Q_{(4,2,1)}$ \\[6pt]
  8 & $64Q_{(8)}+237Q_{(7,1)}+664Q_{(6,2)}+769Q_{(5,3)} +448Q_{(5,2,1)}+384Q_{(4,3,1)}$\\[6pt]
  9 & $93Q_{(9)}+413Q_{(8,1)}+1461Q_{(7,2)}+2441Q_{(6,3)} +1329Q_{(5,4)}+1344Q_{(6,2,1)}+2217Q_{(5,3,1)}+768Q_{(4,3,2)}$\\
  
  \hline
\end{tabular}
\end{table}

\bibliographystyle{abbrv} 
\bibliography{Arxiv_Version}

\end{document}